\documentclass[12pt]{amsart}

\usepackage{amsmath,amssymb}
\usepackage{graphics,epsfig,color}

\newtheorem{theorem}{Theorem}[section]
\newtheorem{corollary}{Corollary}[section]

\newtheorem{remark}{Remark}[section]

\begin{document}

\title[Shape differentiability of Lagrangians]
{Shape differentiability of Lagrangians and application to Stokes problem}

\author{V.A. Kovtunenko$^{\dagger}$ and K. Ohtsuka$^{\ddagger}$}

\thanks{$^{\dagger}$ 
Institute for Mathematics and Scientific Computing, 
Karl-Franzens University of Graz, NAWI Graz, 
Heinrichstr.36, 8010 Graz, Austria; 
Lavrent'ev Institute of Hydrodynamics, 
Siberian Division of Russian Academy of Sciences, 
630090 Novosibirsk, Russia, 
Email: victor.kovtunenko@uni-graz.at}

\thanks{$^{\ddagger}$ 
Faculty of Information Design and Sociology, Hiroshima Kokusai Gakuin University, 
6-20-1, Aki-ku, Hiroshima, 739-0321, Japan, 
Email: ohtsuka@hkg.ac.jp}



\begin{abstract}
A class of convex constrained minimization problems over polyhedral 
cones for geometry-dependent quadratic objective functions 
is considered in a functional analysis framework. 
Shape differentiability of the primal minimization problem 
needs a bijective property for mapping of the primal cone. 
This restrictive assumption is relaxed to bijection of the dual cone 
within the Lagrangian formulation as a primal-dual minimax problem. 
In this paper, we give results on primal-dual shape sensitivity analysis 
that extends the class of shape-differentiable problems 
supported by explicit formula of the shape derivative.
We apply the results to the Stokes problem under mixed Dirichlet--Neumann 
boundary conditions subject to the divergence-free constraint. 
\end{abstract}

\maketitle

\section{Introduction}\label{sec1}

We aim at shape differentiability for a class of convex constrained 
minimization problems over polyhedral cones,  
where the objective functions are assumed quadratic and depend 
on a geometry. 

Typical examples are contact problems in solid mechanics, see 
\cite{KS/97,KO/88}, and other elliptic partial differential equations 
in variable domains with equality and inequality type constraints, 
see \cite{HIKKP/09,KK/14,MKAOA/15,OK/12}. 
Our special interest concerns nonlinear crack problems 
in fracture mechanics due to non-penetration between crack faces, 
which are developed in \cite{IKR/17,IKT/11,KK/00} 
and other works by the authors. 
By this, shape variations may imply regular perturbations 
along a predefined crack path, see \cite{AOK/14,IKRT/12,KOS/02}, 
as well as singular perturbations due to kink of the crack, 
see \cite{KKT/08,KKT/10}. 
A recent result of \cite{KL/16} concerns shape-topological control  
by posing a small defect in the cracked domain. 

From the point of view of shape and topology optimization, 
a shape sensitivity analysis of the problem is performed 
with the help of the velocity method. 
Introducing a proper kinematic velocity, see e.g. \cite{KKR/09}, 
a general perturbation of quadratic constrained minimization problems 
over convex cones in Hilbert spaces is established in \cite{HK/11}. 
An explicit formula of the shape derivative is provided 
by bijective properties of the velocity-based diffeomorphic flow of a geometry. 
However, this result restricts the primal cone to be a bijection within the flow. 
The bijection fails for constraints involving normal on curves 
(e.g. Signorini conditions), having integral, gradient, divergence operator, etc. 
This is rather restrictive, even not a complete list.

In the case of Signorini-type constraints imposed on curvilinear 
manifolds implying cracks, the shape differentiability result is improved in 
\cite{LR/14,Rud/04,Shch/17} relying on a $\Gamma$-convergence of the primal cones. 
For this specific problem, in \cite{Kov/06,KK/07} the assumption of bijection 
is relaxed further to the dual cone within a Lagrangian formalism. 
See another specific example of shape sensitivity of a Lagrangian 
associated with inhomogeneous Dirichlet problem in \cite{DZ/06}, 
and the general Lagrangian method together with related 
primal-dual minimax problems in \cite{IK/08}. 

For other example of such a non-bijective primal cone, in the present work 
we consider a Stokes problem under mixed Dirichlet--Neumann boundary 
conditions subject to the diver\-gence-free constraint. 
We refer to \cite{Cia/13,GR/86,Lad/69} for the Stokes problems,  
and to \cite{BHMS/09,HSS/14} for its shape sensitivity. 
It is worth to stress that the divergence-free constraint is not preserved by transport. 
The treatment of the incompressibility within the dynamical 
shape control of Navier--Stokes equations is discussed in 
\cite[Section~5]{MZ/06}. 
It employs special transforms (Piola transformation, transverse map), 
a hold-all domain assumption, but has a lack of rigorous 
mathematical justification \cite[p.142]{MZ/06}.

In Section~\ref{sec2} we develop  our concept of the shape 
differentiability of Lagrangians in a functional analysis framework. 
Based on the Lagrangian setting which implies a primal-dual minimax problem, 
we relax the bijection assumption from the primal cone 
$K$ (in the space of primal variable) to the dual cone 
$K^\star$ (in the space of dual variable) (see \eqref{19c}). 
This relaxation allows us to lead the primal-dual shape sensitivity analysis 
and to obtain the shape derivative explicitly. 
The improvement of the previous shape sensitivity results 
is attained with respect to non-bijective primal cones,  
thus extending the class of shape-differentiable problems.  

It is important to put our investigation in the classic context 
of optimal value functions adopted in optimization. 
The directional differentiability of optimal value Lagrangians in 
abstract formulation was established in \cite{CS/85} 
(see also \cite[Chapter~4.3.2]{BS/00}), and extended to the 
shape optimization framework in \cite{DZ/88}. 
For a concept of directional differentiability of metric projections 
onto polyhedric sets corresponding to shape derivatives 
we refer to \cite{LSZ/14} and references therein.  

The abstract optimal value Lagrangian function used for shape optimization 
in a time-dependent domain $\Omega_{t}$ with parameter $t$ 
can be defined by a general map of the form: 
\begin{equation}\label{OVF}
\tag{OVF}
\mathbb{R}\mapsto \mathbb{R},\quad 
t\mapsto \mathcal{L}(u_{t},\lambda_{t}; \Omega_{t}), 
\end{equation}
where a saddle point $(u_{t},\lambda_{t})\in V(\Omega_{t})\times 
K^\star(\Omega_{t})$ satisfies 
\begin{multline}\label{SP}
\tag{SP}
\mathcal{L}(u_{t}, p;\Omega_{t}) 
\le\mathcal{L}(u_{t},\lambda_{t};\Omega_{t}) \le\mathcal{L}(w, \lambda_{t};\Omega_{t})\\
\forall (w,p)\in V(\Omega_{t})\times K^\star(\Omega_{t})
\end{multline}
for a Lagrangian 
\begin{equation}\label{L}
\tag{L}
(u,\lambda)\mapsto \mathcal{L}(u,\lambda; \Omega_{t}): \; 
V(\Omega_{t})\times K^\star(\Omega_{t})\mapsto \mathbb{R},
\end{equation}
defined over topological vector spaces $V(\Omega_{t})$ and 
$K^\star(\Omega_{t})$ (the upper star to be explained later on). 
The aim is to find the {\em directional derivative}: 
\begin{equation}\label{DD}
\tag{DD}
\partial_t \mathcal{L}(u_{t},\lambda_{t}; \Omega_{t}) :=
\lim_{s\to0} {\textstyle\frac{\mathcal{L}(u_{t+s},\lambda_{t+s}; 
\Omega_{t+s}) -\mathcal{L}(u_{t},\lambda_{t};\Omega_{t})}{s}}. 
\end{equation}

Since the perturbed optimal value function 
$\mathcal{L}(u_{t+s},\lambda_{t+s}; \Omega_{t+s})$ in \eqref{DD} 
is given by the perturbed Lagrangian 
\begin{equation}\label{PL}
\tag{PL}
(v,\mu)\mapsto \mathcal{L}(v,\mu; \Omega_{t+s}):\; 
V(\Omega_{t+s})\times K^\star(\Omega_{t+s}) \mapsto \mathbb{R},
\end{equation}
which is defined over $s$-dependent spaces 
$V(\Omega_{t+s})\times K^\star(\Omega_{t+s})$, then the 
usual trick in shape optimization is to use a coordinate transformation 
\begin{equation}\label{CT}
\tag{CT}
\phi_s:\Omega_{t}\mapsto \Omega_{t+s},\quad 
\phi^{-1}_s:\Omega_{t+s}\mapsto \Omega_{t}
\end{equation}
that maps \eqref{PL} to a transformed perturbed Lagrangian 
\begin{equation}\label{TPL}
\tag{TPL}
(s,u,\lambda)\mapsto \mathcal{L}_s(u,\lambda; \Omega_{t}):\; 
\mathbb{R}\times V(\Omega_{t})\times K^\star(\Omega_{t}) \mapsto \mathbb{R}
\end{equation}
over fixed spaces $V(\Omega_{t})\times K^\star(\Omega_{t})$ 
such that $\mathcal{L}_0 =\mathcal{L}$ and 
\begin{equation}\label{BL}
\tag{BL}
\mathcal{L}_s(v\circ\phi_s,\mu\circ\phi_s; \Omega_{t})
=\mathcal{L}(v,\mu; \Omega_{t+s})
\end{equation} 
for all $(v,\mu)\in V(\Omega_{t+s})\times K^\star(\Omega_{t+s})$. 
This needs the fulfillment of bijective property between the function spaces
\begin{multline}\label{BS}
\tag{BS}
[v\mapsto v\circ \phi_s]: 
V(\Omega_{t+s})\mapsto V(\Omega_{t}),\\
[\mu\mapsto \mu\circ \phi_s]: 
K^\star(\Omega_{t+s})\mapsto K^\star(\Omega_{t})
\end{multline}
and allows to rewrite \eqref{DD} in the equivalent form: 
\begin{equation}\label{DD'}
\tag{DD'}
\partial_s \mathcal{L}_s(0,u_{t},\lambda_{t}; \Omega_{t}) 
=\lim_{s\to0} {\textstyle\frac{\mathcal{L}_s 
(u_{t+s}\circ\phi_s,\lambda_{t+s}\circ\phi_s; \Omega_{t}) 
-\mathcal{L}(u_{t},\lambda_{t};\Omega_{t})}{s}}. 
\end{equation}
The bijection \eqref{BS} is central in this work. 

In the constrained optimization context, $K^\star$ is associated 
to a dual cone compared with its primal counterpart $K$. 
For the divergence-free constraint, in Section~\ref{sec3} 
we give an example of the space $K^\star(\Omega_{t+s})$ 
where the bijection of dual cones (see \eqref{19c}) fails. 
Namely, considering Stokes problem under no-slip Dirichlet condition, 
the integral identity $\int_{\Omega_{t+s}} v(y) \,dy =0$ 
characterizing the space $L^2_0(\Omega_{t+s})$ (see \eqref{L20}) 
is not preserved by the transport $y =\phi_s(x)$ in general, 
thus,  the equivalence between \eqref{DD} and \eqref{DD'} is not true. 
A possible remedy is to use special area-preserving maps. 
In the current paper, we suggest to consider the Stokes problem 
under mixed Dirichlet--Neumann boundary conditions such that 
the bijection property \eqref{BS} holds true. 

\section{Shape derivative of Lagrangians for polyhedral cones}\label{sec2}

We start the investigation with a family of {\em time-dependent geometric sets} 
$t\mapsto \Omega_{t} \subset \mathbb{R}^d$, $d\in\mathbb{N}$. 

For every fixed time $t\in\mathbb{R}$, we consider two 
{\em geometry-dependent Hilbert spaces} $V(\Omega_{t})$ and $H(\Omega_{t})$ 
with the {\em dual spaces} $V^\star(\Omega_{t})$ and $H^\star(\Omega_{t})$. 
Let a linear operator $A: V(\Omega_{t})\mapsto 
V^\star(\Omega_{t})$ be {\em strongly monotone} such that 
\begin{equation}\label{1}
\langle Au, u \rangle_{\Omega_{t}}\ge \underline{c}_A 
\|u\|^2_{V(\Omega_{t})},\quad \underline{c}_A>0,\quad u\in V(\Omega_{t})
\end{equation}
with the duality pairing $\langle \,\cdot\,, \,\cdot\, \rangle_{\Omega_{t}}$ 
between $V^\star(\Omega_{t})$ and $V(\Omega_{t})$, 
and {\em continuous} such that 
\begin{equation}\label{2}
\|Au\|_{V^\star(\Omega_{t})}\le \overline{c}_A \|u\|_{V(\Omega_{t})},
\quad \overline{c}_A\ge \underline{c}_A>0,\quad u\in V(\Omega_{t})
\end{equation}
uniformly in a time interval $t\in(t_0,t_1)$ with fixed $t_0<t_1$. 
Let a linear operator $B: V(\Omega_{t})\mapsto H(\Omega_{t})$ 
be {\em surjective} 
(i.e. for every $\zeta\in H(\Omega_{t})$ there is at least one $u\in V(\Omega_{t})$ 
such that $Bu =\zeta$) 
and {\em continuous} with the following estimate  
\begin{equation}\label{3}
\|Bu\|_{H(\Omega_{t})}\le \overline{c}_B \|u\|_{V(\Omega_{t})},
\quad \overline{c}_B>0,\quad u\in V(\Omega_{t})
\end{equation}
that holds uniformly for all $t\in(t_0,t_1)$. 

Using the order relation for measured functions in $H(\Omega_{t})$, 
we define the {\em primal cone} as a polyhedral cone as follows 
\begin{equation}\label{4}
K(\Omega_{t}) :=\{u\in V(\Omega_{t}) \vert \quad Bu\ge0\}
\end{equation}
which is convex and closed. 
For a stationary right-hand side $f$ such that 
$f\in \bigcap_{t\in(t_0,t_1)}V^\star(\Omega_{t})$, 
let the {\em geometry-dependent objective function} 
$\mathcal{E}: V(\Omega_{t})\mapsto \mathbb{R}$ be given by 
\begin{equation}\label{5}
\mathcal{E}(u;\Omega_{t}) :=
\langle {\textstyle\frac{1}{2}} Au -f, u \rangle_{\Omega_{t}}
\end{equation}
that is quadratic, bounded due to \eqref{2}, and coercive due to \eqref{1}. 

We consider the {\em primal constrained minimization problem}: 
Find $u_{t}\in K(\Omega_{t})$ such that 
\begin{equation}\label{6}
\mathcal{E}(u_{t};\Omega_{t}) =\min_{w\in K(\Omega_{t})} 
\mathcal{E}(w;\Omega_{t}).
\end{equation}
The unique solution to \eqref{6} exists and satisfies 
the first order {\em optimality condition} in the form of 
a variational inequality due to \eqref{5} and \eqref{6}: 
\begin{equation}\label{7}
\langle Au_{t} -f, w -u_{t} \rangle_{\Omega_{t}}\ge0
\quad\forall w\in K(\Omega_{t}) 
\end{equation}
which is a necessary and sufficient condition for \eqref{6}. 
For a general theory of pseudo-monotone variational inequalities 
see \cite{OG/15}. 

Now we define the {\em dual cone} (in the space of dual variable) as follows  
\begin{equation}\label{8}
K^\star(\Omega_{t}) :=\{\lambda\in H^\star(\Omega_{t}) \vert 
\quad (\lambda, Bu)_{\Omega_{t}}\ge0\quad\forall u\in K(\Omega_{t})\}
\end{equation}
where $( \,\cdot\,, \,\cdot\, )_{\Omega_{t}}$ stands for 
the duality pairing between $H^\star(\Omega_{t})$ and $H(\Omega_{t})$. 
It is important to note that, due to surjection of $B$, 
the dual cone in \eqref{8} can be restated equivalently in the form  
\begin{equation}\label{8'}
\tag{2.8'}
K^\star(\Omega_{t}) =\{\lambda\in H^\star(\Omega_{t}) \vert 
\quad (\lambda, \zeta)_{\Omega_{t}}\ge0\quad \forall \zeta\in H(\Omega_{t}),
\; \zeta\ge0\}.
\end{equation}
The corresponding {\em primal-dual minimax problem} reads: 
Find the pair $(u_{t},\lambda_{t})\in V(\Omega_{t})\times K^\star(\Omega_{t})$ 
such that 
\begin{equation}\label{9}
\mathcal{L}(u_{t},\lambda_{t};\Omega_{t}) =\min_{w\in V(\Omega_{t})} 
\max_{p\in K^\star(\Omega_{t})} \mathcal{L}(w, p;\Omega_{t})
\end{equation}
with  the {\em Lagrangian function} 
$\mathcal{L}: V(\Omega_{t})\times H^\star(\Omega_{t})\mapsto \mathbb{R}$ given by 
\begin{equation}\label{10}
 \mathcal{L}(u,\lambda;\Omega_{t}) :=\mathcal{E}(u;\Omega_{t})
-(\lambda, Bu)_{\Omega_{t}}. 
\end{equation}
Well-posedness and optimality properties of \eqref{9} 
are gathered in the following theorem. 

\begin{theorem}\label{theo1}
(i) There exists a solution of the minimax problem \eqref{9} which implies that 
$(u_{t},\lambda_{t})\in V(\Omega_{t})\times K^\star(\Omega_{t})$ 
is a saddle point: 
\begin{multline}\label{9'}
\tag{2.9'}
\mathcal{L}(u_{t}, p;\Omega_{t}) 
\le\mathcal{L}(u_{t},\lambda_{t};\Omega_{t}) \le\mathcal{L}(w, \lambda_{t};\Omega_{t})\\
\forall (w,p)\in V(\Omega_{t})\times K^\star(\Omega_{t})
\end{multline}
and satisfies the {\em primal-dual optimality conditions}: 
\begin{subequations}\label{11}
\begin{equation}\label{11a}
\langle Au_{t} -f, w \rangle_{\Omega_{t}} -(\lambda_{t}, B w)_{\Omega_{t}} =0
\quad\forall w\in V(\Omega_{t}) 
\end{equation}
\begin{equation}\label{11b}
(p -\lambda_{t}, Bu_{t} )_{\Omega_{t}} \ge0
\quad\forall p\in K^\star(\Omega_{t}). 
\end{equation}
\end{subequations}
The primal component $u_{t}\in K(\Omega_{t})$ 
is unique solution of the primal problem \eqref{6}. 
If the {\em Ladyzhenskaya--Babu\v{s}ka--Brezzi (LBB) condition} 
holds for $\lambda\in H^\star(\Omega_{t})$: 
\begin{equation}\label{12}
\sup_{u\in V(\Omega_{t})\slash\{0\}} {\textstyle\frac{ 
(\lambda, Bu)_{\Omega_{t}}}{\|u\|_{V(\Omega_{t})}}}
\ge \underline{c}_B \|\lambda\|_{H^\star(\Omega_{t})},
\quad 0<\underline{c}_B\le \overline{c}_B
\end{equation}
then the dual component $\lambda_{t}$ is unique. 

(ii) The optimal value objective function $t\mapsto\mathcal{E}(u_{t};\Omega_{t})$ 
defined by \eqref{6} and the optimal value Lagrangian function 
$t\mapsto\mathcal{L}(u_{t},\lambda_{t};\Omega_{t})$ given in \eqref{9} are equal: 
\begin{equation}\label{ovf}
\min_{w\in V(\Omega_{t})} \mathcal{E}(w;\Omega_{t}) =\min_{w\in V(\Omega_{t})} 
\max_{p\in K^\star(\Omega_{t})} \mathcal{L}(w, p;\Omega_{t}).
\end{equation}
\end{theorem}

\begin{proof}
Indeed, based on \eqref{1}--\eqref{10}, existence of a solution 
to the minimax problem follows from e.g. \cite[Theorem~3.11]{KO/88}. 
The inclusion $u_{t}\in K(\Omega_{t})$ is a consequence of the bipolar theorem, 
see e.g. \cite[Theorem~14.1]{Roc/70}, due to surjection of $B$.
The optimality conditions \eqref{11} and the uniqueness assertion 
under LBB condition \eqref{12} are stated e.g. in \cite[Theorem~3.14]{KO/88}. 
The cone $K^\star(\Omega_{t})$ is convex and $V(\Omega_{t})$ is linear, 
the Lagrangian $\mathcal{L}$ is convex-concave and G\^ateaux differentiable, 
so that \eqref{11} is equivalent to (see \cite[Proposition~1.5]{ET/76}): 
\begin{equation*}
\begin{split}
&\langle \partial_{u} \mathcal{L}(u_{t},\lambda_{t};\Omega_{t}), 
w\rangle_{\Omega_{t}} =0 \quad\forall w\in V(\Omega_{t}),\\ 
&( \partial_{\lambda} \mathcal{L}(u_{t},\lambda_{t};\Omega_{t}), 
p -\lambda_{t})_{\Omega_{t}}\ge0 \quad\forall p\in K^\star(\Omega_{t}), 
\end{split}
\end{equation*}
and the pair $(u_{t},\lambda_{t}) \in V(\Omega_{t})\times K^\star(\Omega_{t})$ 
also satisfies \eqref{9'} implying the saddle point (see \cite[Definition~1.1]{ET/76}). 

To proof the assertion (ii), we test \eqref{11b} with $p =0$ and $p =2\lambda_{t}$ 
yielding $(\lambda_{t}, Bu_{t} )_{\Omega_{t}} =0$, hence 
$\mathcal{E}(u_{t};\Omega_{t}) =\mathcal{L}(u_{t},\lambda_{t};\Omega_{t})$ 
in turn implying \eqref{ovf}.
\end{proof}

In the following we lead a shape sensitivity analysis of the problem.

\subsection{Primal-dual shape sensitivity analysis}\label{sec2.1}

For fixed $t\in(t_0,t_1)$ and a small perturbation parameter 
$s\in(t_0-t, t_1-t)$, let given vector-functions 
\begin{subequations}\label{13}
\begin{equation}\label{13a}
[s\mapsto \phi_s],[s\mapsto\phi^{-1}_s]\in 
C^{1}([t_0-t, t_1-t];W^{1,\infty}_{\rm loc}(\mathbb{R}^d;\mathbb{R}^d))
\end{equation}
associate the {\em coordinate transformation} $y=\phi_s(x)$ and 
the inverse mapping $x=\phi^{-1}_s(y)$ such that its composition satisfies: 
\begin{equation}\label{13b}
(\phi^{-1}_s\circ \phi_s)(x)=x,\quad (\phi_s\circ \phi^{-1}_s)(y)=y. 
\end{equation}
\end{subequations}
Then the {\em shape perturbation} 
\begin{equation}\label{14}
\Omega_{t+s} :=\{y\in \mathbb{R}^d \vert \quad y=\phi_s(x),\; x\in \Omega_{t}\}
\end{equation}
builds the {\em diffeomorphism} 
\begin{equation}\label{15}
\phi_s:\Omega_{t}\mapsto \Omega_{t+s}, x\mapsto y;\quad 
\phi^{-1}_s:\Omega_{t+s}\mapsto \Omega_{t}, y\mapsto x.
\end{equation}

We reset the {\em perturbed primal constrained minimization problem}: 
Find $u_{t+s}\in K(\Omega_{t+s})$ such that 
\begin{equation}\label{16}
\mathcal{E}(u_{t+s};\Omega_{t+s}) =\min_{v\in K(\Omega_{t+s})} 
\mathcal{E}(v;\Omega_{t+s})
\end{equation}
and the corresponding {\em perturbed primal-dual minimax problem}: 
Find the pair $(u_{t+s},\lambda_{t+s})\in V(\Omega_{t+s})\times K^\star(\Omega_{t+s})$ 
such that 
\begin{equation}\label{17}
\mathcal{L}(u_{t+s},\lambda_{t+s};\Omega_{t+s})\\ 
=\min_{v\in V(\Omega_{t+s})} 
\max_{\mu\in K^\star(\Omega_{t+s})} \mathcal{L}(v,\mu;\Omega_{t+s})
\end{equation}
with the {\em perturbed Lagrangian and objective functions}, respectively: 
\begin{subequations}\label{18}
\begin{equation}\label{18a}
 \mathcal{L}(v,\mu;\Omega_{t+s}) =\mathcal{E}(v;\Omega_{t+s})
-(\mu, Bv)_{\Omega_{t+s}} 
\end{equation}
\begin{equation}\label{18b}
\mathcal{E}(v;\Omega_{t+s}) =
\langle {\textstyle\frac{1}{2}} Av -f, v \rangle_{\Omega_{t+s}}.
\end{equation}
\end{subequations}
They are defined for $v\in V(\Omega_{t+s})$ and 
$\mu\in H^\star(\Omega_{t+s})$ 
with the duality pairings $\langle \,\cdot\,, \,\cdot\, \rangle_{\Omega_{t+s}}$ 
between $V^\star(\Omega_{t+s})$ and $V(\Omega_{t+s})$, and 
$( \,\cdot\,, \,\cdot\, )_{\Omega_{t+s}}$ between 
$H^\star(\Omega_{t+s})$ and $H(\Omega_{t+s})$. 

Within the kinematic flow \eqref{13}--\eqref{15}, we employ 
the assumptions: The map $[v\mapsto v\circ \phi_s]$ 
is bijective in the function spaces
\begin{subequations}\label{19}
\begin{equation}\label{19a}
V(\Omega_{t+s})\mapsto V(\Omega_{t}),\quad
V^\star(\Omega_{t+s})\mapsto V^\star(\Omega_{t}),
\end{equation}
\begin{equation}\label{19b}
H(\Omega_{t+s})\mapsto H(\Omega_{t}),\quad
H^\star(\Omega_{t+s})\mapsto H^\star(\Omega_{t}),
\end{equation}
and $[\mu\mapsto \mu\circ \phi_s]$ is bijective in the dual cones 
\begin{equation}\label{19c}
K^\star(\Omega_{t+s})\mapsto K^\star(\Omega_{t}).
\end{equation}
As $s\to0$,  let the asymptotic representations hold for the operator $A$:
\begin{equation}\label{19d}
\langle Av, \chi \rangle_{\Omega_{t+s}} =\langle [A +s A^1 +A^2_s] 
(v\circ\phi_s), \chi\circ\phi_s \rangle_{\Omega_{t}}
\end{equation}
with linear bounded operators 
$A^1, A^2_s: V(\Omega_{t})\mapsto V^\star(\Omega_{t})$ 
and the residual $A^2_s$ such that 
\begin{equation}\label{19e}
\| A^2_s u \|_{V^\star(\Omega_{t})}\le c_{RA}(s) 
\| u \|_{V(\Omega_{t})},\quad 0\le c_{RA}(s)={\rm o}(s);
\end{equation}
for the operator $B$:
\begin{equation}\label{19f}
(\mu, Bv )_{\Omega_{t+s}} =(\mu\circ\phi_s, 
[B +s B^1 +B^2_s] (v\circ\phi_s) )_{\Omega_{t}}
\end{equation}
with linear bounded operators 
$B^1, B^2_s: V(\Omega_{t})\mapsto H(\Omega_{t})$ 
such that $B +s B^1 +B^2_s$ is surjective 
and the residual $B^2_s$ satisfies  
\begin{equation}\label{19g}
\| B^2_s u \|_{H(\Omega_{t})}\le c_{RB}(s) 
\| u \|_{V(\Omega_{t})},\quad 0\le c_{RB}(s)={\rm o}(s);
\end{equation}
and for the right-hand side $f$:
\begin{equation}\label{19h}
\langle f, v \rangle_{\Omega_{t+s}} =\langle f +s f^1 +f^2_s, 
v\circ\phi_s \rangle_{\Omega_{t}}
\end{equation}
with $f^1, f^2_s\in V^\star(\Omega_{t})$ and the residual $f^2_s$ such that 
\begin{equation}\label{19i}
\| f^2_s \|_{V^\star(\Omega_{t})}\le c_{Rf}(s),\quad 
0\le c_{Rf}(s)={\rm o}(s)
\end{equation}
for test-functions $v,\chi\in V(\Omega_{t+s})$, 
$\mu\in H^\star(\Omega_{t+s})$, $u\in V(\Omega_{t})$, 
uniformly for all $s\in(t_0-t, t_1-t)$ and $t\in(t_0,t_1)$. 
\end{subequations}

\begin{theorem}\label{theo2}
Under the assumptions \eqref{19}, the {\em optimal value function} 
$\mathbb{R}\mapsto \mathbb{R}, t\mapsto \mathcal{E}(u_{t};\Omega_{t})$ 
of the objective $ \mathcal{E}$ given in \eqref{5} and \eqref{6} 
is {\em shape differentiable} such that 
\begin{equation}\label{20}
{\textstyle\frac{d}{dt}} \mathcal{E}(u_{t};\Omega_{t}) :=
\lim_{s\to0} {\textstyle\frac{\mathcal{E}(u_{t+s};\Omega_{t+s}) 
-\mathcal{E}(u_{t};\Omega_{t})}{s}} 
=\mathcal{L}^1(u_{t},\lambda_{t};\Omega_{t})
\end{equation}
with the {\em shape derivative} $\mathcal{L}^1(u_{t},\lambda_{t};\Omega_{t})$ 
determined by 
\begin{subequations}\label{21}
\begin{equation}\label{21a}
\mathcal{L}^1(u,\lambda;\Omega_{t}) :=
\mathcal{E}^1(u;\Omega_{t}) -(\lambda, B^1 u )_{\Omega_{t}}
\end{equation}
\begin{equation}\label{21b}
\mathcal{E}^1(u;\Omega_{t}) :=
\langle {\textstyle\frac{1}{2}} A^1 u -f^1, u \rangle_{\Omega_{t}}.
\end{equation}
\end{subequations}
\end{theorem}

\begin{proof}
We apply to \eqref{17} the asymptotic formula \eqref{19d}, \eqref{19f}, 
\eqref{19h} and use the assumptions \eqref{19a}--\eqref{19c} to get 
the transformed solution pair $(u_{t+s}\circ\phi_s, \lambda_{t+s}\circ\phi_s)\in 
V(\Omega_{t})\times K^\star(\Omega_{t})$ which solves the minimax problem  
\begin{equation}\label{22}
\mathcal{L}_s(u_{t+s}\circ\phi_s,\lambda_{t+s}\circ\phi_s;\Omega_{t})\\ 
=\min_{w\in V(\Omega_{t})} \max_{p\in K^\star(\Omega_{t})} 
\mathcal{L}_s(w, p;\Omega_{t})
\end{equation}
implying a saddle point (see \eqref{9'}):   
\begin{multline}\label{22'}
\tag{2.23'}
\mathcal{L}_s(u_{t+s}\circ\phi_s, p;\Omega_{t}) 
\le\mathcal{L}_s(u_{t+s}\circ\phi_s,\lambda_{t+s}\circ\phi_s;\Omega_{t})\\ 
\le\mathcal{L}_s(w, \lambda_{t+s}\circ\phi_s;\Omega_{t})\quad 
\forall (w, p)\in V(\Omega_{t})\times K^\star(\Omega_{t}).
\end{multline}
The {\em transformed Lagrangian} 
$\mathcal{L}_s: V(\Omega_{t})\times H(\Omega_{t})\mapsto \mathbb{R}$
is defined via 
\begin{subequations}\label{23}
\begin{equation}\label{230}
\mathcal{L}_s(v\circ\phi_s,\mu\circ\phi_s; \Omega_{t})
=\mathcal{L}(v,\mu; \Omega_{t+s})\quad (\text{with }\mathcal{L}_0 =\mathcal{L})
\end{equation} 
for all $(v,\mu)\in V(\Omega_{t+s})\times K^\star(\Omega_{t+s})$, 
and yields the expansion
\begin{equation}\label{23a}
 \mathcal{L}_s(u,\lambda;\Omega_{t})\\ 
:=\mathcal{L}(u,\lambda;\Omega_{t})
+s \mathcal{L}^1(u,\lambda;\Omega_{t}) +\mathcal{L}^2_s(u,\lambda;\Omega_{t})
\end{equation}
where the first asymptotic terms $\mathcal{L}^1(u,\lambda;\Omega_{t})$ 
is given in \eqref{21a}, and the residual  
\begin{equation}\label{2b}
\mathcal{L}^2_s(u,\lambda;\Omega_{t}) 
:=\langle {\textstyle\frac{1}{2}} A^2_s u -f^2_s, u\rangle_{\Omega_{t}} 
-(\lambda, B^2_s u )_{\Omega_{t}}. 
\end{equation}
\end{subequations}

Based on Theorem~\ref{theo1}, optimality conditions for \eqref{22} are 
\begin{subequations}\label{24}
\begin{multline}\label{24a}
\langle [A +s A^1 +A^2_s] (u_{t+s}\circ\phi_s) 
-(f +s f^1 +f^2_s), w \rangle_{\Omega_{t}}\\
 -(\lambda_{t+s}\circ\phi_s, [B +s B^1 +B^2_s] w)_{\Omega_{t}} =0\quad
\forall w\in V(\Omega_{t}) 
\end{multline}
\begin{multline}\label{24b}
\bigl( p -\lambda_{t+s}\circ\phi_s, 
[B +s B^1 +B^2_s] (u_{t+s}\circ\phi_s) \bigr)_{\Omega_{t}} \ge0\\
\forall p\in K^\star(\Omega_{t}). 
\end{multline}
\end{subequations}
Taking the test function $w =u_{t+s}\circ\phi_s$ in \eqref{24a}, 
using the complementarity 
\begin{equation}\label{25}
(\lambda_{t+s}\circ\phi_s, [B +s B^1 +B^2_s] u_{t+s}\circ\phi_s)_{\Omega_{t}} =0
\end{equation}
which follows from \eqref{24b}, the strong monotony \eqref{1} of $A$, 
and the residual estimates \eqref{19e}, \eqref{19g}, \eqref{19i}, 
for $|s| \in(0,s_0)$ with sufficiently small $s_0>0$ 
and $t\in(t_0,t_1)$ we get the uniform estimate: 
\begin{subequations}\label{26}
\begin{equation}\label{26a}
\| u_{t+s}\circ\phi_s\|_{V(\Omega_{t})} \le{\rm const}.
\end{equation}
Similarly, from \eqref{24a} we derive the uniform estimate in the dual space: 
\begin{equation}\label{26b}
\| \lambda_{t+s}\circ\phi_s\|_{H^\star(\Omega_{t})} \le{\rm const}
\end{equation}
\end{subequations}
for $|s|\in(0,s_1)$ with sufficiently small $0<s_1\le s_0$ and $t\in(t_0,t_1)$. 

From \eqref{26} it follows the existence of $(\overline{u}, \overline{\lambda})\in 
V(\Omega_{t})\times H^\star(\Omega_{t})$ 
and a subsequence denoted by $s_k$ such that as $s_k\to0$:
\begin{subequations}\label{27}
\begin{equation}\label{27a}
u_{t+s_k}\circ\phi_{s_k} \rightharpoonup \overline{u}
\quad\text{weakly in $V(\Omega_{t})$} 
\end{equation}
\begin{equation}\label{27b}
\lambda_{t+s_k}\circ\phi_{s_k} \rightharpoonup \overline{\lambda}
\quad\text{$\star$-weakly in $H^\star(\Omega_{t})$}. 
\end{equation}
Every linear and continuous operator $B$ is weak-to-weak continuous 
(see \cite[Theorem~3.10]{Bre/10}), therefore  
\begin{equation}\label{27c}
B(u_{t+s_k}\circ\phi_{s_k}) \rightharpoonup B\overline{u}
\quad\text{weakly in $H(\Omega_{t})$}. 
\end{equation}
\end{subequations}
In accordance with \eqref{19c} the inclusion 
$\lambda_{t+s}\circ\phi_s\in K^\star(\Omega_{t})$ holds, 
the convex closed set $K^\star(\Omega_{t})$ is $\star$-weakly closed, 
hence $\overline{\lambda}\in K^\star(\Omega_{t})$. 
Since a quadratic form is weakly lower semi-continuous, 
we pass to the limit in \eqref{22'}
using the weak convergences in \eqref{27} and get
\begin{multline*}
\mathcal{L}(\overline{u}, p;\Omega_{t})
\le \liminf_{s_k\to0} \mathcal{L}_{s_k}(u_{t+s_k}\circ\phi_{s_k}, p;\Omega_{t})
\le \mathcal{L}(\overline{u},\overline{\lambda};\Omega_{t})\\
\le \limsup_{s_k\to0} \mathcal{L}_{s_k}(w,\lambda_{t+s_k}\circ\phi_{s_k};\Omega_{t}) 
\le\mathcal{L}(w,\overline{\lambda};\Omega_{t}) 
\end{multline*}
for arbitrary $(w,p)\in V(\Omega_{t})\times K^\star(\Omega_{t})$. 
Therefore, $(\overline{u}, \overline{\lambda}) =(u_{t},\lambda_{t})$ is 
a saddle point satisfying \eqref{9'}, thus solves \eqref{9}. 

In order to estimate the solution difference in the norm, 
we start with the inequality \eqref{1} and rearrange the terms such that  
\begin{multline*}
{\textstyle\frac{\underline{c}_A}{2}} 
\|u_{t+s}\circ\phi_s -u_{t} \|^2_{V(\Omega_{t})} 
\le {\textstyle\frac{1}{2}} \langle A \bigl( u_{t+s}\circ\phi_s -u_{t} \bigr), 
u_{t+s}\circ\phi_s -u_{t} \rangle_{\Omega_{t}}\\ 
=-\langle A \bigl( u_{t+s}\circ\phi_s -u_{t} \bigr), u_{t} \rangle_{\Omega_{t}} 
-{\textstyle\frac{1}{2}} \langle A u_{t}, u_{t} \rangle_{\Omega_{t}} 
+{\textstyle\frac{1}{2}} \langle A \bigl( u_{t+s}\circ\phi_s \bigr), 
u_{t+s}\circ\phi_s \rangle_{\Omega_{t}}\\
=-\langle A \bigl( u_{t+s}\circ\phi_s -u_{t} \bigr), u_{t} \rangle_{\Omega_{t}} 
+\langle f, u_{t+s}\circ\phi_s -u_{t} \rangle_{\Omega_{t}}\\
+\mathcal{L}(u_{t+s}\circ\phi_s,\lambda_{t+s}\circ\phi_s;\Omega_{t}) 
-\mathcal{L}(u_{t},\lambda_{t};\Omega_{t})
+(\lambda_{t+s}\circ\phi_s, [s B^1 +B^2_s] u_{t+s}\circ\phi_s)_{\Omega_{t}}
\end{multline*}
due to the orthogonality relations $(\lambda_{t}, B u_{t} )_{\Omega_{t}} =0$ and \eqref{25}. 
Using further 
\begin{multline*}
\limsup_{s_k\to0} \bigl\{ \mathcal{L}(u_{t+s_k}\circ\phi_{s_k}, 
\lambda_{t+s_k}\circ\phi_{s_k};\Omega_{t}) 
-\mathcal{L}(u_{t},\lambda_{t};\Omega_{t})\bigr\}\\ 
=\limsup_{s_k\to0} \bigl\{ \mathcal{L}_{s_k}(u_{t+s_k}\circ\phi_{s_k}, 
\lambda_{t+s_k}\circ\phi_{s_k};\Omega_{t}) 
-\mathcal{L}_{s_k}(u_{t},\lambda_{t+s_k}\circ\phi_{s_k};\Omega_{t}) \bigr\} \le0
\end{multline*}
because of \eqref{22'} with $w =u_{t}$ and \eqref{27}, we conclude that 
\begin{subequations}\label{28}
\begin{equation}\label{28a}
{\textstyle\frac{\underline{c}_A}{2}} \limsup_{s_k\to0} 
\|u_{t+s_k}\circ\phi_{s_k} -u_{t} \|^2_{V(\Omega_{t})}\le 0.
\end{equation}
Therefore, from \eqref{3} it follows that as $s_k\to0$ 
\begin{equation}\label{28b}
\| B(u_{t+s_k}\circ\phi_{s_k} -u_{t}) \|_{H(\Omega_{t})} \to0.
\end{equation}
From \eqref{11a} and \eqref{24a} we arrive at 
\begin{equation*}
(\lambda_{t+s}\circ\phi_s -\lambda_{t}, B w )_{\Omega_{t}}
=\langle A (u_{t+s}\circ\phi_s -u_{t}), w \rangle_{\Omega_{t}} +{\rm O}(s)
\end{equation*}
for all $w\in V(\Omega_{t})$, 
henceforth the surjection of $B$ provides that 
\begin{equation}\label{28c}
\|\lambda_{t+s_k}\circ\phi_{s_k} -\lambda_{t} \|_{H^\star(\Omega_{t})} \to0.
\end{equation}
\end{subequations}
The relations \eqref{28} imply the strong convergences in \eqref{27}. 

Based on the asymptotic formula \eqref{23} 
we find the lower bound: 
\begin{subequations}\label{29}
\begin{multline}\label{29a}
\mathcal{L}_s(u_{t+s}\circ\phi_s,\lambda_{t+s}\circ\phi_s;\Omega_{t}) 
-\mathcal{L}(u_{t},\lambda_{t};\Omega_{t})\\
\ge\mathcal{L}_s(u_{t+s}\circ\phi_s,\lambda_{t};\Omega_{t}) 
-\mathcal{L}(u_{t+s}\circ\phi_s,\lambda_{t};\Omega_{t})\\
=s \mathcal{L}^1(u_{t+s}\circ\phi_s,\lambda_{t};\Omega_{t}) 
+\mathcal{L}^2_s(u_{t+s}\circ\phi_s,\lambda_{t};\Omega_{t})
\end{multline}
using the maximum in \eqref{22'} with the test function $p =\lambda_{t}$, and  
the minimum in \eqref{9'} with the test function $w =u_{t+s}\circ\phi_s$.  
Similarly, we calculate the upper bound: 
\begin{multline}\label{29b}
\mathcal{L}_s(u_{t+s}\circ\phi_s,\lambda_{t+s}\circ\phi_s;\Omega_{t}) 
-\mathcal{L}(u_{t},\lambda_{t};\Omega_{t})\\
\le\mathcal{L}_s(u_{t},\lambda_{t+s}\circ\phi_s;\Omega_{t}) 
-\mathcal{L}(u_{t}, 
\lambda_{t+s}\circ\phi_s;\Omega_{t})\\
=s \mathcal{L}^1(u_{t},\lambda_{t+s}\circ\phi_s;\Omega_{t}) 
+\mathcal{L}^2_s(u_{t},\lambda_{t+s}\circ\phi_s;\Omega_{t}) 
\end{multline}
\end{subequations}
utilizing the minimum in \eqref{22'} with the test function $w =u_{t}$, 
and the maximum in \eqref{9'} with the test function $p =\lambda_{t+s}\circ\phi_s$.  
The strong convergences \eqref{28} provide the asymptotic order of the residuals: 
\begin{equation*}
\mathcal{L}^2_{s_k}(u_{t},\lambda_{t+s_k}\circ\phi_{s_k};\Omega_{t}) 
={\rm o}(s_k),\quad \mathcal{L}^2_{s_k}(u_{t+s_k}\circ\phi_{s_k}, 
\lambda_{t};\Omega_{t}) ={\rm o}(s_k)
\end{equation*}
hence from \eqref{29} divided with $s$ it follows existence of the limit 
\begin{equation}\label{30}
\lim_{s_k\to0} {\textstyle\frac{\mathcal{L}(u_{t+s_k}, 
\lambda_{t+s_k};\Omega_{t+s_k}) 
-\mathcal{L}(u_{t},\lambda_{t};\Omega_{t})}{s_k}} 
=\mathcal{L}^1(u_{t},\lambda_{t};\Omega_{t})
\end{equation}
because of the identity $\mathcal{L}(u_{t+s},\lambda_{t+s};\Omega_{t+s}) =\mathcal{L}_s(u_{t+s}\circ\phi_s,\lambda_{t+s}\circ\phi_s;\Omega_{t})$ 
due to \eqref{230}. 
The optimal value Lagrangian and objective functions are equal, 
see \eqref{ovf} and the similar identity  
$\mathcal{L}(u_{t+s},\lambda_{t+s};\Omega_{t+s}) 
=\mathcal{E}(u_{t+s};\Omega_{t+s})$, 
then \eqref{30} coincides with formula \eqref{20} of the shape derivative 
and completes the proof. 
\end{proof}

\begin{remark}\label{rem1}
Theorem~\ref{theo2} presents a direct proof of the shape differentiability. 
Since the bijection \eqref{19a}--\eqref{19c} holds, then the Correa--Seeger 
theorem  on directional differentiability can be applied by checking 
hypotheses (H1)--(H4) in \cite[Chapter~10, Theorem~5.1]{DZ/11}. 
\end{remark}

To formulate the hypotheses, let us define the optimal values 
\begin{equation*}
l_t :=\sup_{p\in K^\star(\Omega_{t})}  \inf_{w\in V(\Omega_{t})} \mathcal{L}(w, p;\Omega_{t}) \le \inf_{w\in V(\Omega_{t})} 
\sup_{p\in K^\star(\Omega_{t})}  \mathcal{L}(w, p;\Omega_{t})=: l^t,
\end{equation*}
and the solution sets 
\begin{multline*}
V_t =\{u\in V(\Omega_{t})\vert\; \sup_{p\in K^\star(\Omega_{t})} 
\mathcal{L}(u, p;\Omega_{t}) = l^t\},\\ 
K^\star_t =\{\lambda\in K^\star(\Omega_{t})\vert\; \inf_{w\in V(\Omega_{t})} 
\mathcal{L}(w, \lambda;\Omega_{t}) = l_t\}\quad\text{for }t\in(t_0,t_1). 
\end{multline*}

(H1) The solution sets are nonempty due to Theorem~\ref{theo1}. 
Moreover, $l_t =l^t$ and $V_t =\{u_{t}\}$, $K^\star_t =\{\lambda_{t}\}$ 
are singleton.
 
(H2) For $t\in(t_0,t_1)$ there exists the partial derivative:
\begin{subequations}\label{CS}
\begin{multline}\label{CSa}
\lim_{s\to0} {\textstyle\frac{\mathcal{L}_s 
(u,\lambda; \Omega_{t}) -\mathcal{L}(u,\lambda;\Omega_{t})}{s}} 
=\mathcal{L}^1(u,\lambda;\Omega_{t})\\
\forall (u,\lambda)\in \bigl(\cup_{\tau\in(t_0,t_1)} V_\tau\times K^\star_t\bigr) 
\cup \bigl(V_t\times \cup_{\tau\in(t_0,t_1)} K^\star_\tau\bigr)
\end{multline} 
within the asymptotic expansion \eqref{23a} 
which is uniform with respect to $(u,\lambda)$. 
This hypothesis holds due to assumptions \eqref{19d}--\eqref{19i}. 

(H3) There exist an accumulation point $\overline{u}\in V_t$ and 
a subsequence $u_{t+s_k}\circ\phi_{s_k}\in V_t$ denoted by $s_k$ such that 
\begin{equation}\label{CSb}
\|u_{t+s_k}\circ\phi_{s_k} -\overline{u}\|_{V(\Omega_{t})}\to0 
\quad\text{as }s_k\to0,
\end{equation}
which is proved in \eqref{27a} with $\overline{u} =u_{t}$, and 
\begin{equation}\label{CSc}
\liminf_{s_k\to0} \mathcal{L}^1 (u_{t+s_k}\circ\phi_{s_k},p; \Omega_{t})
\ge \mathcal{L}^1 (\overline{u},p; \Omega_{t})
\quad\forall p\in K^\star_t, 
\end{equation}
that holds due to continuity in the strong topology 
of the bilinear mapping $w\mapsto \mathcal{L}^1 (w,p; \Omega_{t})$. 

(H4) There exist an accumulation point $\overline{\lambda}\in K^\star_t$ 
and a subsequence $\lambda_{t+s_k}\circ\phi_{s_k}\in K^\star_t$ 
denoted by $s_k$ such that 
\begin{equation}\label{CSd}
\|\lambda_{t+s_k}\circ\phi_{s_k} -\overline{\lambda}\|_{H^\star(\Omega_{t})}\to0 
\quad\text{as }s_k\to0,
\end{equation}
with $\overline{\lambda} =\lambda_{t}$ according to \eqref{27c}, and 
\begin{equation}\label{CSe}
\limsup_{s_k\to0} \mathcal{L}^1 (w,\lambda_{t+s_k}\circ\phi_{s_k}; \Omega_{t})
\le \mathcal{L}^1 (w,\overline{\lambda}; \Omega_{t})
\quad\forall w\in V_t, 
\end{equation}
\end{subequations}
provided by the weak continuity of the linear mapping 
$p\mapsto \mathcal{L}^1 (w,p; \Omega_{t})$. 

Indeed, testing \eqref{22'} with $(w,p) =(u_{t},\lambda_{t})$ and 
\eqref{9'} with $(w,p) =(u_{t+s}\circ\phi_s,\lambda_{t+s}\circ\phi_s)$ gives 
\begin{multline*}
{\textstyle\frac{\mathcal{L}_s(u_{t+s}\circ\phi_s,\lambda_{t};\Omega_{t}) 
-\mathcal{L}(u_{t+s}\circ\phi_s,\lambda_{t};\Omega_{t})}{s}}
\le{\textstyle\frac{\mathcal{L}_s(u_{t+s}\circ\phi_s,\lambda_{t+s}\circ\phi_s;\Omega_{t}) 
-\mathcal{L}(u_{t},\lambda_{t};\Omega_{t})}{s}} =:\Delta(s)\\
\le {\textstyle\frac{\mathcal{L}_s(u_{t},\lambda_{t+s}\circ\phi_s;\Omega_{t}) 
-\mathcal{L}(u_{t},\lambda_{t+s}\circ\phi_s;\Omega_{t})}{s}}.
\end{multline*}
Since we show that the expansion \eqref{23a} holds, we get 
\begin{multline*}
\mathcal{L}^1 (u_{t+s}\circ\phi_s,\lambda_{t}; \Omega_{t})
+{\textstyle\frac{1}{s}} \mathcal{L}^2_s (u_{t+s}\circ\phi_s,\lambda_{t}; \Omega_{t})
\le\Delta(s)\\
\le \mathcal{L}^1 (u_{t},\lambda_{t+s}\circ\phi_s; \Omega_{t})
+{\textstyle\frac{1}{s}} \mathcal{L}^2_s (u_{t},\lambda_{t+s}\circ\phi_s; \Omega_{t})
\end{multline*}
and use \eqref{CSc} and \eqref{CSe} to pass it to the limit as $s_k\to0$, 
which is essentially the idea of the theorem of Correa--Seeger. 

\begin{remark}\label{rem2}
The assumptions \eqref{19d}--\eqref{19i} on the asymptotic expansion 
can be relaxed in Theorem~\ref{theo2} to the abstract conditions \eqref{CS}. 
\end{remark}

We note the important special cases in two corollaries. 
The first corollary relates the assumption \eqref{19c} of the dual cones 
to the primal cones, see \cite[Theorem~3.4]{HK/11}. 

\begin{corollary}\label{cor1}
If the primal cone \eqref{4} is such that $K^\star(\Omega_{t}) =K(\Omega_{t})$, 
then the assumption \eqref{19c} is equivalent to bijection of the primal cones 
\begin{equation}\label{31}
K(\Omega_{t})\mapsto K(\Omega_{t+s}) 
\end{equation} 
and formula of the shape derivative \eqref{20} implies the equality 
\begin{equation*}
{\textstyle\frac{d}{dt}} \mathcal{E}(u_{t};\Omega_{t}) 
=\mathcal{E}^1(u_{t};\Omega_{t}),\quad 
(\lambda_{t}, B^1 u_{t} )_{\Omega_{t}} =0
\end{equation*}
under the assumptions \eqref{19} used in Theorem~\ref{theo2}. 
\end{corollary}

The second corollary extends the result to equality constraints. 

\begin{corollary}\label{cor2}
The inequality constraint in \eqref{4} can be replaced 
with the equality constraint resulting in the following primal and dual cones 
\begin{equation}\label{33}
K(\Omega_{t}) =\{u\in V(\Omega_{t}) \vert \; Bu =0\},
\quad K^\star(\Omega_{t}) =H^\star(\Omega_{t}). 
\end{equation}
Then the assumption \eqref{19c} is satisfied within \eqref{19b}, 
thus Theorem~\ref{theo2} holds true under the made assumptions. 
\end{corollary}

In the next section we realize an application of Corollary~\ref{cor2} 
to the Stokes problem with the divergence-free equality constraint, 
that mapping is not a bijection again. 

\section{Example of shape derivative: Stokes problem}\label{sec3}

Let $\Omega_{t}$ be a domain with Lipschitz continuous boundary, 
denote by $n_{t}$ the outward unit normal vector, 
and let the boundary $\partial\Omega_{t}$ consist of two disjoint sets 
$\Gamma^D_{t}$ and $\Gamma^N_{t}$. 
For a given stationary external force 
$f\in H^1_{\rm loc}(\mathbb{R}^d;\mathbb{R}^d)$, 
we consider the {\em Stokes problem} finding a vector-valued field of flow velocity 
$u_{t} =((u_{t})_1,\dots,(u_{t})_d)$ and a scalar-valued $\lambda_{t}$ 
implying the pressure such that 
\begin{subequations}\label{S}
\begin{equation}\label{Sa}
-\Delta u_{t} +\nabla \lambda_{t} =f\quad\text{in $\Omega_{t}$}
\end{equation}
\begin{equation}\label{Sb}
{\rm div} u_{t} =0\quad\text{in $\Omega_{t}$}
\end{equation}
\begin{equation}\label{Sc}
u_{t} =0 \quad\text{on $\Gamma^D_{t}$}
\end{equation}
\begin{equation}\label{Sd}
{\textstyle\frac{\partial}{\partial n_{t}}} u_{t} -\lambda_{t} n_{t} =0 
\quad\text{on $\Gamma^N_{t}$}.
\end{equation}
\end{subequations}
The mixed boundary conditions imply no-slip \eqref{Sc} 
and a Neumann-type condition \eqref{Sd}. 
For mixed boundary conditions appropriate for the Stokes equation 
see \cite{BMMW/10}, \cite[Chapter~6]{KMR/01}. 

Corresponding to \eqref{S} primal minimization problem reads: 
Find $u_{t} \in V(\Omega_{t})$ such that ${\rm div}u_{t} =0$ and 
\begin{equation}\label{34}
\mathcal{E}(u_{t};\Omega_{t}) =\min_{w\in K(\Omega_{t})} 
\mathcal{E}(w;\Omega_{t}).
\end{equation}
minimizing the objective function of the energy: 
\begin{equation}\label{35}
\mathcal{E}(w;\Omega_{t}) =
\int_{\Omega_{t}} \sum_{i=1}^d ({\textstyle\frac{1}{2}} 
|\nabla w_i|^2 -f_i w_i) \,dx
\end{equation}
over the primal cone determined by the divergence-free constraint:  
\begin{equation}\label{36}
K(\Omega_{t}) =\{w\in V(\Omega_{t})\vert 
\quad {\rm div} w =0\;\text{a.e. $\Omega_{t}$}\}
\end{equation}
in the function space 
\begin{equation}\label{35'}
V(\Omega_{t}) =\{w\in H^1(\Omega_{t};\mathbb{R}^d)\vert 
\quad w =0\;\text{a.e. $\Gamma^D_{t}$}\}. 
\end{equation} 
The operators $A=-\Delta$ and $B ={\rm div}$ constitute 
the respective duality pairings: 
\begin{subequations}\label{37}
\begin{equation}\label{37a}
\langle Au, w \rangle_{\Omega_{t}} =\int_{\Omega_{t}} \sum_{i=1}^d 
(\nabla u_i)^\top \nabla w_i \,dx,\quad u,w\in V(\Omega_{t})
\end{equation}
\begin{equation}\label{37b}
( \lambda, B u )_{\Omega_{t}} =\int_{\Omega_{t}} 
\lambda {\rm div}u \,dx,\quad \lambda\in H(\Omega_{t}),
\end{equation}
\end{subequations}
and the dual cone 
\begin{equation}\label{36'}
K^\star(\Omega_{t}) =\{\lambda\in H^\star(\Omega_{t})\vert \quad 
( \lambda, B u )_{\Omega_{t}} =0\quad \forall u\in K(\Omega_{t})\},
\end{equation}
where $H(\Omega_{t}) =H^\star(\Omega_{t}) 
=L^2(\Omega_{t};\mathbb{R})$.

If the surface measure ${\rm meas} (\Gamma^N_{t})>0$, then
the LBB condition \eqref{12} holds \cite[Theorem~7.2]{KO/88}, 
which means that $B: V(\Omega_{t}) \mapsto H(\Omega_{t})$ 
is surjective and $K^\star(\Omega_{t})=H^\star(\Omega_{t})$. 
So we can apply Corollary \ref{cor2}. 

If ${\rm meas} (\Gamma^N_{t}) =0$, then $B: H^1_0(\Omega_{t}; 
\mathbb{R}^d) \mapsto L^2_0(\Omega_{t};\mathbb{R})$, where 
\begin{equation}\label{L20}
L^2_0(\Omega_{t};\mathbb{R}) =\{
\lambda\in H(\Omega_{t})\vert\quad ( \lambda, 1 )_{\Omega_{t}} =0\}
\end{equation}
and its dual space excludes constants. 
In this case we cannot apply Corollary \ref{cor2}.  
In fact, the bijection in \eqref{19c} between $L^2_0(\Omega_{t};\mathbb{R})$ and 
\begin{equation*}
L^2_0(\Omega_{t+s};\mathbb{R}) =\{
\mu\in H(\Omega_{t+s})\vert\quad ( \mu, 1 )_{\Omega_{t+s}} =0\}
\end{equation*}
fails because $( \mu, 1 )_{\Omega_{t+s}} \not= ( \mu\circ\phi_s, 1 )_{\Omega_{t}}$ 
according to the transformation formula \eqref{19f}. 

The primal-dual formulation of \eqref{34} consists in  
finding the pair $(u_{t},\lambda_{t})\in V(\Omega_{t}) 
\times L^2(\Omega_{t};\mathbb{R})$ which is a saddle-point: 
\begin{equation}\label{38}
\mathcal{L}(u_{t},\lambda_{t};\Omega_{t}) =\min_{w\in V(\Omega_{t})} 
\max_{p\in L^2(\Omega_{t};\mathbb{R})} 
\mathcal{L}(w, p;\Omega_{t})
\end{equation}
of the Lagrangian 
\begin{equation}\label{39}
\mathcal{L}(w, p;\Omega_{t}) =\mathcal{E}(w;\Omega_{t})
-\int_{\Omega_{t}} p {\rm div} w \,dx
\end{equation}
where the dual cone $K^\star(\Omega_{t}) =L^2(\Omega_{t};\mathbb{R})$ 
according to \eqref{36'}.
The optimality conditions \eqref{11} 
for the problems \eqref{35} and \eqref{39} have the form: 
\begin{subequations}\label{40}
\begin{equation}\label{40a}
\int_{\Omega_{t}} \bigl(\sum_{i=1}^d ( \nabla (u_{t})_i^\top 
\nabla w_i -f_i w_i) -\lambda_{t} {\rm div} w \bigr) \,dx =0
\quad\forall w\in V(\Omega_{t}) 
\end{equation}
\begin{equation}\label{40b}
\int_{\Omega_{t}} p\, {\rm div}u_{t} \,dx =0
\quad\forall p\in L^2(\Omega_{t};\mathbb{R}).
\end{equation}
\end{subequations}
The solution pair is unique 
since the LBB condition holds in this case. 

For a stationary {\em kinematic velocity} 
$\Lambda\in W^{1,\infty}_{\rm loc}(\mathbb{R}^d;\mathbb{R}^d)$ 
the unique solutions $[s\mapsto \phi_s], [s\mapsto \phi^{-1}_s]\in 
C^{1}([-T,T];W^{1,\infty}_{\rm loc}(\mathbb{R}^d;\mathbb{R}^d))$ 
of the {\em autonomous ODE systems} with some $T>0$: 
\begin{equation*}
\left\{ \begin{array}{rl}
{\textstyle\frac{d}{ds}}\phi_s =\Lambda(\phi_s)
&\;\text{for $s\not=0$}\\[1ex]
\phi_s=x&\;\text{for $s=0$},
\end{array}
\right.\quad 
\left\{ \begin{array}{rl}
{\textstyle\frac{d}{ds}}\phi^{-1}_s =-\Lambda(\phi^{-1}_s)
&\;\text{for $s\not=0$}\\[1ex]
\phi^{-1}_s =y&\;\text{for $s=0$}
\end{array}
\right.
\end{equation*}
satisfy \eqref{13} and build the diffeomorphism \eqref{15}, 
see \cite[Lemma~2.2]{HK/11}. 
In the non-stationary case, the velocity 
$\Lambda\in C([-T,T];W^{1,\infty}_{\rm loc}(\mathbb{R}^d;\mathbb{R}^d))$ 
is defined by $\Lambda(t+s,y) ={\textstyle\frac{d}{ds}}\phi_s(\phi^{-1}_s(y))$, 
see \cite[Section~2.9]{SZ/92}.
By this, the {\em transformation matrix} 
$\nabla_y\phi^{-1}_s :=
\{ (\phi^{-1}_s)_{i,j} \}_{i,j=1}^d$, where 
$(\phi^{-1}_s)_{i,j} =\frac{\partial (\phi^{-1}_s)_i}{\partial y_j}$,  
and the Jacobian determinant ${\rm det}(\nabla \phi_s)$ 
of the matrix $\nabla\phi_s :=\{ (\phi_s)_{i,j} \}_{i,j=1}^d$,
where $(\phi_s)_{i,j} =\frac{\partial (\phi_s)_i}{\partial x_j}$,   
admit the following asymptotic expansion as $s\to0$: 
\begin{equation}\label{41}
\nabla_y\phi^{-1}_s(\phi_s) =I -s \nabla\Lambda +r^1_s,\quad 
|\nabla \phi_s| =1 +s {\rm div}\Lambda +r^2_s,
\end{equation}
with the uniform estimate of the residuals 
$\|r^1_s\|_{C([-T,T];L^\infty_{\rm loc}(\mathbb{R}^{d\times d}))} ={\rm o}(s)$ 
and $\|r^2_s\|_{C([-T,T];L^\infty_{\rm loc}(\mathbb{R}))} ={\rm o}(s)$, 
where $\nabla\Lambda =\{ \frac{\partial \Lambda_i}{\partial x_j} \}_{i,j=1}^d$, 
and $I$ stands for the $d$-by-$d$-identity matrix. 

We apply the coordinate transformation $y=\phi_s(x)$ 
to the duality pairings in \eqref{37} rewritten over 
the perturbed domain $\Omega_{t+s}$ according to \eqref{14}.  
As the result, using the chain rule $\nabla_y 
=(\nabla_y \phi^{-1}_s (\phi_s) )^\top \nabla_x$ 
and \eqref{41} we derive the following asymptotic expansions 
corresponding to the assumptions \eqref{19d}--\eqref{19i}. 
Indeed, the operator $A$ is expanded as follows 
for $v, \chi\in H^1(\Omega_{t+s};\mathbb{R}^d)$: 
\begin{subequations}\label{42}
\begin{multline}\label{42a}
\langle Av, \chi \rangle_{\Omega_{t+s}} =\int_{\Omega_{t+s}} 
\sum_{i=1}^d (\nabla_y v_i)^\top \nabla_y \chi_i \,dy\\
=\int_{\Omega_{t}} \sum_{i=1}^d (\nabla (v_i\circ\phi_s))^\top 
\nabla_y \phi^{-1}_s (\phi_s)  (\nabla_y \phi^{-1}_s (\phi_s) )^\top
\nabla (\chi_i\circ\phi_s)) \,{\rm det}(\nabla \phi_s) dx\\
=\int_{\Omega_{t}} \sum_{i=1}^d (\nabla (v_i\circ\phi_s))^\top 
\bigl( I +s \{ ({\rm div}\Lambda) I -\nabla  \Lambda 
-(\nabla  \Lambda)^\top \} \bigr) \nabla (\chi_i\circ\phi_s)) \,dx 
+{\rm o}(s)
\end{multline}
implying \eqref{19d} and \eqref{19e} with the first asymptotic term 
\begin{equation}\label{42b}
\langle A^1 u, w \rangle_{\Omega_{t}}
=\int_{\Omega_{t}} \sum_{i=1}^d (\nabla u_i)^\top 
\{ ({\rm div}\Lambda) I -\nabla  \Lambda 
-(\nabla  \Lambda)^\top \} \nabla w_i \,dx.
\end{equation}
Accordingly, for $\mu\in L^2(\Omega_{t+s};\mathbb{R})$ 
the operator $B$ is expanded as
\begin{multline}\label{42c}
( \mu, B v )_{\Omega_{t+s}} 
=\int_{\Omega_{t+s}} \mu {\rm div}_y v \,dy\\
=\int_{\Omega_{t}} (\mu\circ\phi_s) \sum_{i,j=1}^d 
(\phi^{-1}_s)_{j,i}  (v\circ\phi_s)_{i,j} \,{\rm det}(\nabla \phi_s) \,dx
\end{multline}
which implies \eqref{19f} and \eqref{19g} with 
\begin{equation}\label{42d}
( \lambda, B^1 u )_{\Omega_{t}} 
=\int_{\Omega_{t}}  \lambda \{ ({\rm div}\Lambda) ({\rm div}u) 
- \sum_{i,j=1}^d \Lambda_{j,i} u_{i,j} \} \,dx
\end{equation}
for $u,w\in H^1(\Omega_{t};\mathbb{R}^d)$ and 
$\lambda\in L^2(\Omega_{t};\mathbb{R})$. 
And the transformation 
\begin{multline}\label{42e}
\langle f, v \rangle_{\Omega_{t+s}} 
=\int_{\Omega_{t+s}} \sum_{i=1}^d f_i v_i \,dy\\
=\int_{\Omega_{t}} \sum_{i=1}^d (f_i\circ\phi_s) 
(v_i\circ\phi_s) \,{\rm det}(\nabla \phi_s) dx
\end{multline}
due to \eqref{41} and 
$f_i\circ\phi_s =f_i +s \Lambda^\top\nabla f_i +{\rm o}(s)$ 
follows \eqref{19h} and \eqref{19i} with the first asymptotic term 
\begin{equation}\label{42f}
\langle f^1, u \rangle_{\Omega_{t}} =\int_{\Omega_{t}} 
\sum_{i=1}^d \bigl( ({\rm div} \Lambda) f_i 
+\Lambda^\top\nabla f_i \bigr) u_i \,dx.
\end{equation}
\end{subequations}
The decompositions \eqref{42} agree the assumptions \eqref{19a} and \eqref{19b}. 

The assumption of bijection \eqref{31} is not true for the primal 
cone \eqref{36} because of the transformation of the divergence 
(see formula \eqref{42c}). 
Nevertheless,  the bijection of the dual cone allows us to apply 
Theorem~\ref{theo2} in the form of Corollary~\ref{cor2}. 
The shape differentiability of the Stokes problem based on \eqref{42} 
and using ${\rm div} u_{t} =0$ is established in the next theorem. 

\begin{theorem}\label{theo4}
The Stokes problem given in \eqref{34}--\eqref{36} 
has the {\em shape derivative} 
${\textstyle\frac{d}{dt}} \mathcal{E}(u_{t};\Omega_{t}) 
=\mathcal{L}^1(u_{t},\lambda_{t};\Omega_{t})$ 
which is defined in \eqref{20} and calculated 
according to formula \eqref{21} as follows
\begin{subequations}\label{43}
\begin{equation}\label{43a}
\mathcal{L}^1(u_{t},\lambda_{t};\Omega_{t}) =
\mathcal{E}^1(u_{t};\Omega_{t}) 
+\int_{\Omega_{t}}  \lambda_{t} \sum_{i,j=1}^d 
\Lambda_{j,i} (u_{t})_{i,j} \,dx
\end{equation}
\begin{multline}\label{43b}
\mathcal{E}^1(u_{t};\Omega_{t}) =\int_{\Omega_{t}} 
\sum_{i=1}^d \Bigl( {\textstyle\frac{1}{2}} 
({\rm div} \Lambda) |\nabla (u_{t})_i|^2 
-\sum_{k,j =1}^d (u_{t})_{i,k} \Lambda_{k,j} (u_{t})_{i,j}\\ 
-\bigl( ({\rm div} \Lambda) f_i 
+\Lambda^\top\nabla f_i \bigr) (u_{t})_i \Bigr) \,dx.
\end{multline}
\end{subequations}
\end{theorem}

We remark the singularity at the intersection 
$\overline{\Gamma^D_{t}} \cap\overline{\Gamma^N_{t}}$ (see e.g. \cite{Ben/07}) 
such that $(u_{t},\lambda_{t})$ is generally not in 
$H^2(\Omega_{t};\mathbb{R}^d)\times H^1(\Omega_{t};\mathbb{R})$ 
as shown in \cite[Theorem 1.3.2]{Oht/18}. 
Let the singular points are contained locally in a domain 
$\overline{\omega_{t}}\subset \overline{\Omega_{t}}$ such that 
$(u_{t},\lambda_{t})\in 
H^2(\Omega_{t}\setminus\omega_{t};\mathbb{R}^d)\times 
H^1(\Omega_{t}\setminus\omega_{t};\mathbb{R})$, 
and $f,\Lambda\equiv$ const in $\omega_{t}$. 
In this case, using integration of \eqref{43} by parts we get 
the following expression over the boundary of $\Omega_{t}\setminus\omega_{t}$: 
\begin{multline*}
\mathcal{L}^1(u_{t},\lambda_{t};\Omega_{t}) 
=\int_{\partial(\Omega_{t}\setminus\omega_{t})} \sum_{i=1}^d \Bigl(
(\Lambda^\top n_{t}) \bigl( {\textstyle \frac{1}{2}}|\nabla (u_{t})_i|^2 -f_i (u_{t})_i \bigr)\\  
-(\Lambda^\top\nabla (u_{t})_i) \bigl( {\textstyle\frac{\partial}{\partial n_{t}}} (u_{t})_i 
-\lambda_{t} (n_{t})_i \bigr) \Bigr) \,dS_x,
\end{multline*}
which implies the generalized J-integral (see \cite{AOK/14,Oht/18}). 

In the case of $\Gamma^N_{t} =\emptyset$, to preserve the integral 
(see \eqref{L20}), this needs special area-preserving maps 
that form special linear group $SL(d)$ as stated in the last result. 

\begin{corollary}\label{cor3}
Let the problem \eqref{S} be stated under solely no-slip Dirichlet condition 
$u_{t} =0$ on $\partial\Omega_{t} = \Gamma^D_{t}$. 
If the transformation $y =\phi_s(x)$ is characterized 
by the Jacobian determinant ${\rm det}(\nabla \phi_s) =1$, 
then formula \eqref{43} in Theorem~\ref{theo4} still holds true 
with ${\rm div} \Lambda =0$. 
\end{corollary}

Examples of such area-preserving bijection are translation 
and rotation of bodies obeying circular or cylindrical symmetry 
that maps the body into itself. 

\section{Conclusion}\label{sec4}

The result of the shape sensitivity analysis is useful in structure 
optimization, see e.g. \cite{AJT/04}. 
In particular, a positive/ negative sign of the shape derivative 
forces respectively either increase or decay of the objective 
function $\mathcal{E}$ of the energy.

For further development in the shape differentiability of Lagrangians, 
we may suggest to combine Theorem~\ref{theo2} 
together with Corollary~\ref{cor2} in order to account simultaneously 
for both equality and inequality type constraints within polyhedral cones. 
The example is the Stokes problem under the threshold slip boundary condition, 
see \cite{LRT/07,RR/99}.

\vspace{5mm}\noindent {\bf Acknowledgment}. 
V.A.K. is supported by the Austrian Science Fund (FWF) 
project P26147-N26: "Object identification problems: numerical analysis" (PION) 
and the Austrian Academy of Sciences (OeAW).\\ 
K.O. is supported by the JSPS KAKENHI Grant Number 16K05285.\\
The joint work  began in CoMFoS15 that is the workshop by the Activity group 
MACM (Mathematical Aspects of Continuum Mechanics) of JSIAM. 
The authors thank two referees for the comments 
which helped to improve the manuscript.

\end{document}